\documentclass[10pt]{amsart}
\usepackage{amsmath,amsfonts,amsthm}
\usepackage [dvips]{graphicx}

\theoremstyle{plain}

\newtheorem{theorem}{Theorem}[section]
\newtheorem{lemma}[theorem]{Lemma}

\newtheorem{proposition}{Proposition}
\theoremstyle{remark}

\newtheorem{conjecture}[theorem]{Conjecture}

\newcommand{\RR}{{\mathbb R}}
\newcommand{\QQ}{{\mathbb Q}}

\newcommand{\ZZ}{{\mathbb Z}}
\newcommand{\NN}{{\mathbb N}}

\newcommand{\cF}{{\mathcal F}}

\newcommand{\cP}{{\mathcal P}}

\newcommand{\cW}{{\mathcal W}}

\def\eps{\varepsilon}

\author{Amie Wilkinson}
\address{Northwestern University, Evanston, Illinois, USA.}
\title{Smooth ergodic theory}

\begin{document}
\maketitle

\tableofcontents

\section*{Glossary}

\noindent{\bf Conservative, Dissipative:} Conservative dynamical systems (on a
compact phase space) are those that preserve a finite  measure equivalent to volume.  
Hamiltonian dynamical systems are important examples of conservative systems.  
Systems that are not conservative are called dissipative. Finding 
physically meaningful invariant measures for dissipative maps is 
a central object of study in smooth ergodic theory.

\medskip

\noindent{\bf Distortion estimate:} A key technique in smooth ergodic theory,
a distortion estimate for a smooth map $f$ gives a bound  on
the variation of the jacobian of $f^n$ in a given region, for 
$n$ arbirtarily large.  The jacobian of a smooth map at a point
$x$ is the absolute value of the determinant of derivative at $x$, measured 
in a fixed Riemannian metric. The jacobian measures the distortion of volume under $f$ in that metric.

\medskip

\noindent{\bf Hopf Argument:} A technique developed by Eberhard Hopf for proving
that a conservative diffeomorphism or flow is ergodic.  The argument relies on the Ergodic Theorem
for invertible transformations, the density of continuous functions among integrable
functions, and the existence of stable and unstable foliations for the system.
The argument has been used, with various modifications, to establish ergodicity for
hyperbolic, partially hyperbolic and nonuniformly hyperbolic systems.

\medskip

\noindent{\bf Hyperbolic:}  A compact invariant set $\Lambda\subset M$ for a diffeomorphism
$f\colon M\to M$ is hyperbolic if, at every point in $\Lambda$, the tangent space
splits into two subspaces, one that is uniformly contracted by the derivative
of $f$, and another that is uniformly expanded. Expanding 
maps and Anosov diffeomorphisms are examples of globally hyperbolic maps.
Hyperbolic diffeomorphisms and flows are the archetypical smooth systems 
displaying chaotic behavior, and their dynamical properties are well-understood.  
Nonuniform hyperbolicity and partial hyeprbolicity are two generalizations
of hyperbolicity that encompass a broader class of systems and display
many of the chaotic features of hyperbolic systems.

\medskip

\noindent{\bf Sinai-Ruelle-Bowen (SRB) measure:} The concept of SRB measure is a rigorous
formulation of what it means for an invariant measure to be ``physically meaningful.''
An SRB measure attracts a large set of orbits into its support, and its statistical
features are reflected in the behavior of these attracted orbits.

\section*{Definition and Importance of the Subject}

{\em Smooth ergodic theory} is the study of the statistical and geometric properties of measures invariant
under a smooth transformation or flow.  The study of smooth ergodic theory is as old as 
the study of abstract ergodic theory, having its origins in Bolzmann's Ergodic Hypothesis
in the late 19th Century.  As a response to Boltzmann's hypothesis, which was formulated in the
context of Hamiltonian Mechanics, Birkhoff and von Neumann defined ergodicity in the 1930's and proved
their foundational ergodic theorems.  The study of ergodic properties of smooth systems
saw an advance in the work of Hadamard and E. Hopf in the 1930's their study of geodesic flows
for negatively curved surfaces.  Beginning in the 1950's, Kolmogorov, Arnold and Moser developed
a perturbative theory producing obstructions to ergodicity in Hamiltonian systems, known as 
Kolmogorov-Arnold-Moser (KAM) Theory.  Beginning in the 1960's with the work of Anosov and
Sinai on hyperbolic systems, the study of smooth ergodic theory has seen intense activity.
This activity continues today, as the ergodic properties of systems displaying weak forms of
hyperbolicity are further understood, and KAM theory is applied in increasingly broader contexts.


\section{Introduction}

This entry  focuses on the basic arguments and principles 
in smooth ergodic theory,
illustrating with simple and straightforward examples.  
The classic texts \cite{arnoldavez, mane} are a good supplement.

The discussion here sidesteps the topic of 
Kolmogorov-Arnold-Moser (KAM) theory, which has
played an important role in the development of smooth ergodic theory;
see the entry in this volume.  For reasons of space,
detailed discussion of several active areas in smooth ergodic theory
is omitted, including: 
higher mixing properties (Kolmogorov, Bernoulli, etc.),
finer statistical properties (fast decay of correlations, 
Central Limit Theorem, large deviations),
smooth thermodynamic formalism (transfer operators, 
pressure, dynamical zeta functions, etc.), the smooth
ergodic theory of random dynamical systems, 
as well as any mention of infinite invariant measures. 
The text \cite{baladi} covers many
of these topics, and the texts \cite{kifer1,kifer2,liu} treat
random smooth ergodic theory in depth. An excellent discussion of many 
of the recent developments in the field of smooth ergodic theory is \cite{bdv}.

This entry assumes knowledge of the basic concepts in 
ergodic theory and of basic differential topology.  
The texts \cite{cfs} and \cite{hirsch} contain the necessary background.

\section{The volume class}

For simplicity, assume that $M$ is a compact, boundaryless 
$C^\infty$ Riemannian manifold,
and that $f\colon M\to M$ is an orientation-preserving, 
$C^1$ map satisfying $m(D_x f)>0$,
for all $x\in M$, where
$$m(D_xf) = \inf_{v\in T_xM, \|v\|=1} \|D_xf(v)\|.$$
If $f$ is a diffeomorphism, then this assumption is automatically satisfied, since 
in that case $m(D_xf) = \|D_{f(x)} f^{-1}\|^{-1} > 0$.  For non-invertible maps, this assumption is essential in much of the following discussion. The Inverse Function Theorem implies that
any map $f$ satisfying these hypotheses is a covering map of positive degree $d\geq 1$.

These assumptions will avoid the issues of infinite measures and 
the behavior of $f$ near critical points and singularities of the derivative.  
For most results discussed in this entry, this assumption is not too restrictive.  The existence of critical points and other singularities is, however, a complication that cannot be avoided in many important applications.  The ergodic-theoretic analysis of such examples can be considerably more involved, but contains many of the elements discussed in this entry.  The discussion in
Section~\ref{s.critical} indicates
how some of these additional technicalities arise and can be overcome. 
For simplicity, the discussion here is confined almost exclusively to discrete 
time evolution. Many, though not all, of the 
the results mentioned here carry over to flows and semiflows using, 
for example, a cross-section construction (see \cite{mane}, Chapter 1).

Every smooth map $f\colon M\to M$ satisfying these hypotheses preserves a natural  measure {\em class}, the measure class
of a finite, smooth Riemannian volume on $M$.  Fix such a volume $\nu$ on $M$.  Then there exists a
continuous, positive {\em jacobian} function $x\mapsto \hbox{jac}_x f$ on $M$, with the property that
for every sufficiently small ball $B\subset M$, and every measurable set $A\subset B$ one has:
$$\nu(f(A)) = \int_B \hbox{jac}_x f \,d\nu(x).$$ 
The jacobian of $f$ at $x$ is none other than the
absolute value of the determinant of the derivative $D_xf$ (measured in the given Riemannian metric).
To see that the measure class of $\nu$ is preserved by $f$, observe that the Radon-Nikodym derivative
$\frac{d f_\ast\nu}{d\nu}(x)$ at $x$ is equal to $\sum_{y\in f^{-1}(x)} (\hbox{jac}_y f)^{-1} > 0$.
Hence $f_\ast\nu$ is equivalent to $\nu$, and $f$ 
preserves the measure class of $\nu$.

In many contexts, the map $f$ has a natural {\em invariant} measure 
in the measure class of volume. In this case, $f$ is said to be 
{\em conservative}.  One setting in which a natural invariant
smooth measure appears is Hamiltonian dynamics.  Any solution to Hamilton's equations preserves
a smooth volume called the {\em Liouville measure}.  Furthermore, along the invariant, constant energy hypermanifolds
of a Hamiltonian flow, the Liouville measure decomposes smoothly into invariant measures,
each of which is equivalent to the induced Riemannian volume. In this way, many systems of physical or geometric origin, such as billiards, geodesic flows, hard sphere gases, and
evolution of the $n$-body problem give rise to smooth conservative dynamical systems. See the entry
on Dynamics of Hamiltonian Systems.

Note that even though $f$ preserves a smooth measure class, it might not
preserve any measure in that measure class.  Consider, for example, a diffeomorphism
$f\colon S^1\to S^1$ of the circle with exactly two fixed points, $p$ and $q$, $f'(p)>1>f'(q)>0$.
Let $\mu$ be an $f$-invariant probability measure.  Let $I$ be a neighborhood of $p$.
Then $\bigcap_{n=1}^\infty f^{-n}(I) = \{p\}$, but on the other hand, $\mu(f^{-n}(I))=\mu(I)>0$,
for all $n$.  This implies that $\mu(\{p\})>0$, and so $\mu$ does not lie
in the measure class of volume.  This is an example of a dissipative map.  A
map $f$ is called
{\em dissipative} if every $f$-invariant measure with full support
has a singular part with respect to volume. As was just seen, if a 
diffeomorphism $f$ has a periodic sink, then $f$ is dissipative; more 
generally, if a diffeomorphism
$f$ has a periodic point $p$ of period $k$ such that $\hbox{jac}_p f^k \neq 1$, then $f$ is dissipative.

\section{The fundamental questions}

For a given smooth map $f\colon M\to M$, there are 
the following fundamental questions. 

\begin{enumerate}

\item Is $f$ conservative? That is, does there exist an invariant measure in the class of volume? If so, is it unique?

\item When $f$ is conservative, what are its statistical properties?  Is it ergodic,
mixing, a K-system, Bernoulli, etc?  Does it obey a Central Limit Theorem, fast decay of correlations, 
large deviations estimates, etc?

\item If $f$ is dissipative, does there exist an invariant measure, not in the class of volume, but (in some
sense) natural with respect to volume?  What are the statistical properties of such a measure, if it exists?

\end{enumerate}

There are several plausible ways to ``answer'' these questions.  One might fix a given map $f$ of
interest and ask these questions for that specific $f$. What tends to happen 
in the analysis of a single map $f$ is that either:

\begin{itemize}
\item the question can be answered
using ``soft'' methods, and so the answer applies not only to $f$ but to perturbations of $f$,
or even to {\em generic} or {\em typical} $f$ inside a class of maps; or
\item the proof requires ``hard'' analysis or precise asymptotic information
and cannot possibly be answered for a specific 
$f$, but can be answered for a large set of $f_t$ in a typical
(or given) parametrized family $\{f_t\}_{t\in (-1,1)}$ of smooth maps containing $f=f_0$.
\end{itemize}
Both types of results appear in the discussion that follows.

\section{Lebesgue measure and local properties of volume}

Locally, any measure in the measure class of volume is, 
after a smooth change of coordinates,
equivalent to Lebesgue measure in $\RR^n$.  In fact, more is true: Moser's Theorem implies that 
locally any Riemannian volume
is, after a smooth change of coordinates, {\em equal} to Lebesgue measure in $\RR^n$. 
Hence to study many of the local properties of volume, it suffices
to study the same properties for Lebesgue measure.  

One of the basic properties of Lebesgue measure is that every set of positive Lebesgue measure
can be approximated arbitrarily well in measure from the outside by an open set, and from
the inside by a compact set.  A consequence of this property, of fundamental importance
in smooth ergodic theory, is the following statement.

\medskip

\noindent{\em  {\bf \em Fundamental Principle $\#$1:} Two disjoint, positive Lebesgue measure sets cannot mix together uniformly at all scales.}

\medskip

As an illustration of this principle, consider the following elementary exercise in measure theory. First, some notation. If $\nu$ is a measure and 
$A$ and $B$ are $\nu$-measurable sets with $\nu(B) >0$,
the {\em density of $A$ in $B$} is defined by: 
 $$
 \nu(A\colon B) = \frac{\nu(A\cap B)}{\nu(B)}.
 $$ 

\begin{proposition}\label{p.density} 
Let ${\mathcal P}_1, {\mathcal P}_2, \ldots$ be sequence of (mod 0) finite partitions of the circle
$S^1$ into open intervals, with the properties: a) any element of ${\mathcal P}_n$ is a (mod 0) union of elements of ${\mathcal P}_{n+1}$, and b) the maximum diameter of elements of ${\mathcal P}_n$ tends to $0$ as $n\to\infty$.

Let $A$ be any set of positive Lebesgue measure in $S^1$.
Then there exists a sequence of intervals $I_1, I_2, \ldots$, with $I_n\in {\mathcal P}_n$
such that:
$$ \lim_{n\to\infty}\lambda(A\colon I_n)=1.$$
\end{proposition}

\begin{proof} Assume that Lebesgue measure has been normalized
so that $\lambda(S^1) = 1$. Fix a (mod 0) cover of 
$S^1\setminus A$ by pairwise disjoint elements \{$J_i$\} 
of the union $\bigcup_{n=1}^\infty {\mathcal P}_n$ with the properties:
$$\lambda(J_1)\geq \lambda(J_2) \geq \cdots, \hbox{and}$$ 
$$\lambda(\bigcup_{i=1}^\infty J_i) = \sum_{i=1}^\infty \lambda(J_i) < 1.$$
For $n\in \NN$, let $U_n$ be the union of all the intervals $J_i$ that are contained in ${\mathcal P}_n$, and let $V_n = \bigcup_{i=1}^n U_n$. This
defines an increasing sequence of natural numbers $i_1 = 1 < i_2 < i_3 < \cdots$ such that
$U_n = \bigcup_{i=i_n}^{i_{n+1}-1}J_n$ and $V_n = \bigcup_{i=1}^{i_{n+1}-1} J_n$.
  
For each $n$, the interval $I_n$ will be chosen 
to be an element ${\mathcal P_n}$, disjoint from $V_n$,
 in which the density of $\bigcup_{i=n+1}^\infty U_{i}$ 
is very small (approaching $0$ as
$n\to \infty$).  Since $(S^1\setminus A)\cap I_n$
is contained in $\bigcup_{i=n+1}^\infty U_{i}$, this choice of 
$I_n$ will ensure that the 
density of $A$ in $I_n$ is large (approaching $1$ as $n\to \infty$).  

To make this choice of $I_n$, note first that
the density of $\bigcup_{i=n+1}^\infty U_{i}$ inside of $S^1\setminus V_n$ is: 
$$\frac{\lambda(\bigcup_{i=n+1}^\infty U_{i} )}{\lambda(S^1\setminus V_n)} = 
\frac{\sum_{i=i_{n+1}}^\infty \lambda(J_i)}{1- \sum_{i=1}^{i_{n+1}-1} \lambda(J_i)}
 =a_n.$$
Note that, since $\sum_{i=1}^\infty \lambda(J_i) < 1,$ one has 
$a_n\to 0$ as $n\to \infty$.
Since the density of $\bigcup_{i=n+1}^\infty U_{i}$ inside of $S^1\setminus V_n$ is
at most $a_n$,  there is an interval $I_n$ in ${\mathcal P}_n$, disjoint from $V_n$, such that the density of
$\bigcup_{i=n+1}^\infty U_{i}$ inside of $I_n$ is at most $a_n$.  Then 
$$\lim_{n\to\infty}\lambda(A\colon  I_n) \geq \lim_{n\to\infty} 1-a_n = 1.$$
\end{proof}

In smooth ergodic theory, it is often useful to use a variation on Proposition~\ref{p.density} (generally,
in higher dimensions) in which the partitions ${\mathcal P}_n$ are nested, dynamically-defined partitions.
A simple application of this method can be used to prove that the doubling map on the circle
is ergodic with respect to Lebesgue measure, which is done in Section~\ref{s.ergexamples}.

Notice that this proposition does not claim that the intervals 
$I_n$ are nested.  If one imposes
stronger conditions on the partitions ${\mathcal P}_n$, then one can draw stronger conclusions.

A very useful theorem in this respect is the Lebesgue Density Theorem.
A point $x\in M$ is a {\em Lebesgue density point} 
of a measurable set $X\subseteq M$ if
$$\lim_{r\to 0} m(X\colon  B_r(x)) = 1,$$
where $B_r(x)$ is the Riemannian ball of radius $r$ centered at $x$.
Notice that the notion of Lebesgue density point
depends only on the smooth structure of $M$, because
any two Riemannian metrics have the same Lebesgue density points.
 The Lebesgue Density Theorem states that if $A$ is a measurable
set and $\widehat A$ is the set of Lebesgue density points of $A$,
then $m(A\,\Delta\, \widehat A) = 0$.

\section{Ergodicity of the basic examples}\label{s.ergexamples}

This section contains proofs of 
the ergodicity of two basic examples of conservative
smooth maps: irrational rotations on the circle and the doubling map on the circle.  See the entry
Ergodic Theory: Basic Examples and Constructions for a more detailed description of
these maps.  These proofs serve as an elementary illustration of some of the fundamental techniques
and principles in smooth ergodic theory.

\bigskip

\noindent{\bf Rotations on the circle.} Denote by $S^1$ the circle $\RR/\ZZ$, which is an additive group,
and by $\lambda$ normalized Lebesgue-Haar measure on $S^1$.  Fix a real number $\alpha\in\RR$.  The rotation $R_\alpha\colon S^1\to S^1$ is the translation defined by $R_\alpha(x) = x+\alpha$.
Since translations preserve Lebesgue-Haar measure, the map $R_\alpha$ is conservative.  Note that $R_\alpha$ is
a diffeomorphism and an isometry with respect to the canonical flat metric (length) on $S^1$.

\begin{proposition} If $\alpha\notin \QQ$, then the rotation $R_\alpha\colon S^1\to S^1$ is ergodic
with respect to Lebesgue measure.
\end{proposition}

\begin{proof}  Let $A$ be an $R_\alpha$-invariant set in $S^1$, and suppose that $0 <\lambda(A) < 1$. Denote by
$A^c$ the complement of $A$ in $S^1$.
Fix $\eps>0$.  Proposition~\ref{p.density} implies that there exists an interval $I\subset S^1$ such that 
the density of $A$ in $I$ is large: $\lambda(A\colon I) > 1 - \eps$.  
Similarly,
one may choose an interval $J$ such that $\lambda(A^c\colon  J) > 1 - \eps$.
Without loss of generality, one may choose $I$ and $J$ to have the same length.
Since $\alpha$ is irrational, $R_\alpha$ has a dense orbit, which meets the interval $I$.
Since $R_\alpha$ is an isometry, this impies that there is an integer $n$ such 
that $\lambda(R_\alpha^n(I)\,\Delta\, J) < \eps \lambda(I)$.  Since $\lambda(I)=\lambda(J)$, this readily
implies that $|\lambda(A\colon  R_\alpha^n(I)) - \lambda(A\colon J)| < \eps$. Also, since $A$ is invariant,
and $R_\alpha$ is invertible and preserves measure, 
one has: $$\lambda(A\colon  R_\alpha^n(I)) = \lambda(R_\alpha^n(A)\colon  R_\alpha^n(I)) = \lambda(A\colon I)>1-\eps.$$
But for $\eps$ sufficiently small, this contradicts the facts that
$\lambda(A\colon J) = 1- \lambda(A^c\colon J) < \eps$ and $|\lambda(A\colon  R_\alpha^n(I)) - \lambda(A\colon J)| < \eps$.
\end{proof}

Note that this is not a proof of 
the strongest possible statement about $R_\alpha$ (namely, minimality and unique ergodicity).  The point here is to show how ``soft'' arguments are often sufficient to establish ergodicity; this proof uses 
no more about $R_\alpha$ than the fact that it is a transitive isometry.
Hence the same argument shows:
\begin{theorem} Let $f\colon M\to M$ be a transitive isometry of a Riemannian manifold $M$.  Then $f$ is ergodic with respect to Riemannian volume.
\end{theorem}

One can isolate from this proof a useful principle:
\medskip

\noindent{\em  {\bf \em Fundamental Principle $\#$2:} Isometries preserve Lebesgue density at all scales, for arbitrarily many iterates.}

\medskip

This principle implies, for example, that a smooth action by a 
compact Lie group on $M$ is ergodic along
typical (nonsingular) orbits.  This principle
is also useful in studying area-preserving flows on surfaces and,
in a refined form, unipotent flows on homogeneous spaces.  
In the case of surface flows, ergodicity questions
can be reduced to a study of interval exchange transformations.  See the 
entries Basic Examples and Billiards in this volume for a detailed discussion
of interval exchange transformations and flows on surfaces.  The entry
on Ergodic Theory of Group Actions contains detailed information
on unipotent flows.

\bigskip

\noindent{\bf Doubling map on the circle.} Let $T_2\colon S^1\to S^1$ be the doubling map defined by $T_2(x)= 2x$.  Then $T_2$ is a degree-2 covering map and endomorphism of $S^1$ with constant jacobian
$\hbox{jac}_xT_2 \equiv 2$. Since $\frac{d(T_2)_\ast \lambda}{d\lambda} = \frac12 + \frac12 = 1$, $T_2$ preserves
Lebesgue-Haar measure.  The doubling map is the simplest example of a hyperbolic dynamical system,
a topic treated in depth in the next section.

As with the previous example, the focus here is
on the property of ergodicity.  It is again 
possible to prove much stronger results about $T_2$,
such as Bernoullicity, by other methods.  Instead, here is
a soft proof of ergodicity that will generalize
readily to other contexts. 
\begin{proposition}\label{p.T_2ergodic} The doubling map $T_2\colon S^1\to S^1$ is ergodic with respect to Lebesgue measure.
\end{proposition}

\begin{proof}  Let $A$ be a $T_2$-invariant set in $S^1$ with $\lambda(A)>0$. Let $p\in S^1$ be the
fixed point of $T_2$, so that $T_2(p)=p$. 
For each $n\in\NN$, the preimages of $p$ under $T_2^{-n}$ define
a (mod 0) partition ${\mathcal P}_n$ into $2^n$ open intervals of length $2^{-n}$; 
the elements of ${\mathcal P}_n$ are the connected components of $S^1\setminus T_2^{-n}(\{p\})$. 
Note that the sequence of partitions
${\mathcal P}_1, {\mathcal P}_2,\ldots$ is nested, in the sense of Proposition~\ref{p.density}.
Restricted to any interval $J\in {\mathcal P}_n$, 
the map $T_2^n$ is a diffeomorphism onto $S^1\setminus\{p\}$ with constant jacobian 
$\hbox{jac}_x(T_2^n) = (T_2^n)'(x) = 2^n$. 

Since $A$ is invariant, it follows that $T_2^{-n}(A) = A$. 
Fix $\eps>0$. Proposition~\ref{p.density} implies
that there exists an $n\in\NN$ and an interval $J\in {\mathcal P}_n$ such that $\lambda(A\colon J)>1-\eps$.
Note that $T_2^{n}(A\cap J)\subset A$.  But then
\begin{eqnarray*}
\lambda(A) &\geq& \lambda(T_2^{n}(A\cap J))\\
&=& \int_{A\cap J} \hbox{jac}_x(T_2^n)\, d\lambda(x) \\
&= & 2^n \lambda(A\cap J)\\
& = & 2^n \lambda(A\colon J)\lambda(J)\\
& > & 2^n (1-\eps) \lambda(J) = 1-\eps.
\end{eqnarray*}
Since $\eps$ was arbitrary, one obtains that $\lambda(A)=1$.
\end{proof}

In this proof, the facts that the intervals in ${\mathcal P}_n$ have
constant length $2^{-n}$ and that the jacobian of $T_2^n$ restricted
to such an interval is constant and equal to $2^n$ are not essential. 
The key fact really used in this proof is the assertion that the ratio:
$$\frac{\lambda(T_2^{n}(A\cap J)\colon  T_2^{n}(J))}{\lambda(A\colon J)}$$
is bounded, {\em independently of $n$}. In this case, the ratio is $1$ for all $n$
because $T_2$ has constant jacobian.  

It is tempting to try to extend this proof to other expanding maps on the circle, for
example, a $C^1$, $\lambda$-preserving map $f\colon S^1\to S^1$ with $d_{C^1}(f,T_2)$ small.
Many of the aspects of this proof carry through {\em mutatis mutandis} for 
such an $f$, save for one. A $C^1$-small perturbation of $T_2$ will
in general no longer have constant jacobian, and the {\em variation} of the jacobian
of $f^n$ on a small interval can be (and often is) unbounded.  The reason
for this unboundedness is a lack of control of the modulus of continuity of $f'$.
Hence this argument can fail for  $C^1$ perturbations of $T_2$.  
On the other hand, the argument still works for {\em $C^2$} 
perturbations of $T_2$, even when the jacobian is not constant.

The principle behind this fact can be loosely summarized:
\medskip

\noindent{\em  {\bf \em Fundamental Principle $\#$3:} On controlled scales, iterates
of $C^2$ expanding maps distort Lebesgue density in a controlled way.}

\medskip

This principle requires further explanation and justification, which will come 
in the 
following section.  The $C^2$ hypothesis in this principle accounts for the fact that
almost all results in smooth ergodic theory assume a $C^2$ hypothesis (or something
slightly weaker).

\section{Hyperbolic systems}

One of the most developed areas of smooth ergodic theory is in the study of hyperbolic
maps and attractors.  This section defines 
hyperbolic maps and attractors, provides examples,
and investigates their ergodic properties.  
See \cite{robinson, katokhasselblatt} and the entry Hyperbolic
Dynamical Systems in this series for a thorough discussion 
of the topological and smooth properties of hyperbolic systems.

A {\em hyperbolic structure} on a compact $f$-invariant set $\Lambda\subset M$ is
given by a $Df$-invariant splitting  $T_\Lambda M = E^u\oplus E^s$ of the tangent bundle over $\Lambda$
and constants $C, \mu>1$ such that, for every $x\in \Lambda$ and $n\in\NN$:
$$v\in E^u(x) \quad \implies \quad \|D_xf^n(v)\| \geq C^{-1}\mu^n \|v\|,\quad \hbox{and}$$
$$v\in E^s(x) \quad \implies \quad \|D_xf^n(v)\| \leq C\mu^{-n} \|v\|.$$
A hyperbolic attractor for a
map $f\colon M\to M$ is given by an open set $U\subset M$ such that: $f(U)\subset \overline{U}$, and
such that the set $\Lambda = \bigcap_{n\geq 0} f^n(U)$ carries a hyperbolic
structure.  The set $\Lambda$ is called the {\em attractor}, and $U$ is an {\em attracting region}.
A map $f\colon M\to M$ is {\em hyperbolic} if  $M$
decomposes (mod 0) into a finite union of attracting regions for hyperbolic attractors.  Typically one assumes as
well that the restriction of $f$ to each attractor $\Lambda_i$ is topologically transitive.

Every point $p$ in a hyperbolic set $\Lambda$ has  smooth {\em stable manifold} $\cW^s(p)$ and {\em unstable manifold}
$\cW^u(p)$, tangent, respectively, to the subspaces $E^s(p)$ and $E^u(p)$. The set $\cW^s(p)$
is precisely the set of $q\in M$ such that $d(f^n(p), f^n(q))$ tends to $0$ as $n\to\infty$, and
it follows that $f(\cW^s(p)) = \cW^s(f(p))$.
When $f$ is a diffeomorphism, the unstable manifold $\cW^u(p)$ is uniquely defined and is the stable manifold
of $f^{-1}$. When $f$ is not invertible, local unstable manifolds exist, but generally are not unique. 
If $\Lambda$ is a transitive hyperbolic attractor, then every unstable manifold of every point 
$p\in\Lambda$ is dense in $\Lambda$.

\bigskip

\noindent{\bf Examples of hyperbolic maps and attractors}

\medskip

\noindent{\em Expanding maps.}  The previous section mentioned briefly the $C^r$ perturbations of the doubling
map $T_2$.  Such perturbations (as well as $T_2$ itself) are examples of {\em expanding maps}.  A map
$f\colon M\to M$ is {\em expanding} if there exist constants $\mu>1$ and $C>0$  such that,
for every $x\in M$, and every nonzero vector $v\in T_xM$:
$$\|D_xf^n(v)\| \geq C\mu^n \|v\|,$$
with respect to some (any) Riemannian metric on $M$.
An expanding map is clearly hyperbolic, with $U=M$, $E^s$ the trivial bundle,  and $E^u = TM$.
Any disk in $M$ is a local unstable manifold for $f$.

\bigskip

\noindent{\em Anosov diffeomorphisms.}  A diffeomorphism $f\colon M\to M$ is called {\em Anosov} 
if the tangent bundle splits as a direct sum $TM = E^u\oplus E^s$ of two $Df$-invariant
subbundles, such that $E^u$ is uniformly expanded and $E^s$ is uniformly contracted
by $Df$.  Similarly, a flow $\varphi_t\colon M\to M$ is called Anosov if the tangent bundle splits 
as a direct sum $TM = E^u\oplus E^0\oplus E^s$ of three $D\varphi_t$-invariant
subbundles, such that $E^0$ is generated by $\dot \varphi$,
$E^u$ is uniformly expanded and $E^s$ is uniformly contracted by $D\varphi_t$.
Like expanding maps, an Anosov diffeomorphism is an Anosov attractor with $\Lambda=U=M$.

A simple example of a conservative Anosov diffeomorphism is a hyperbolic linear automorphism 
of the torus.  Any matrix $A\in SL(n, \ZZ)$ induces an automorphism of $\RR^n$ preserving
the integer lattice $\ZZ^n$, and so descends to an automorphism $f_A\colon  T^n\to T^n$
of the $n$-torus $T^n= \RR^n/\ZZ^n$.  Since the determinant of $A$ is $1$,
the diffeomorphism $f_A$ preserves Lebesgue-Haar measure on $T^n$.
In the case where none of the eigenvalues of $A$ have modulus $1$, the resulting 
diffeomorphism $f_A$ is Anosov. The stable bundle $E^s$ at $x\in T^n$ is the parallel
translate to $x$ of the sum of the contracted generalized eigenspaces of $A$, and
the unstable bundle $E^u$ at $x$ is the translated sum of expanded eigenspaces.

In general, the invariant subbundles $E^u$ and $E^s$ of an Anosov diffeomorphism are integrable and tangent to 
a transverse pair of foliations $\cW^u$ and $\cW^s$, respectively (see, e.g.  \cite{hps} for a proof of this).
The leaves of $\cW^s$ are uniformly contracted by $f$, and the leaves of $\cW^u$
are uniformly contracted by $f^{-1}$. The leaves of these foliations are as 
smooth as $f$, but the tangent bundles to the leaves do not vary smoothly in the manifold.
The regularity properties of these foliations play an important role in
the ergodic properties of Anosov diffeomorphisms.

The first Anosov flows to be studied extensively were the geodesic flows for manifolds of
negative sectional curvatures.  As these flows are Hamiltonian, they are conservative.  
Eberhard Hopf showed in the 1930's that such geodesic flows for surfaces are ergodic with respect to
Liouville measure \cite{hopf}; it was not until the 1960's that ergodicity of all such flows
was proved by Anosov \cite{anosov}.  The next section describes, in
the context of Anosov diffeomorphisms, Hopf's method and important refinements due to Anosov and Sinai.

\bigskip

\noindent{\em DA attractors.} A simple way to produce a non-Anosov hyperbolic attractor on the torus
is to start with an Anosov diffeomorphism, such as a linear hyperbolic automorphism, 
and deform it in a neighborhood
of a fixed point, turning a saddle fixed point into a source, while preserving the
stable foliation.  If this procedure is carried
out carefully enough, the resulting diffeomorphism is a dissipative hyperbolic diffeomorphism,
called a {\em derived from Anosov (DA)} attractor. Other examples of hyperbolic attractors
are the Plykin attractor and the solenoid.  See \cite{robinson}.

\bigskip

\noindent{\bf Distortion estimates}

\medskip

Before describing the ergodic properties of hyperbolic systems, it is 
useful to pause for a brief discussion of distortion estimates.
Distortion estimates are behind almost every result in smooth ergodic theory.  In the hyperbolic setting,
distortion estimates are applied to the action of $f$ on unstable manifolds to show
that the volume distortion of $f$ along unstable manifolds can be controlled for arbitrarily
many iterates.

The example mentioned at the end of the previous section illustrates the ideas in a distortion estimate. Suppose that $f\colon S^1\to S^1$ is a $C^2$ expanding map, such as a $C^2$ small
perturbation of $T_2$.  Then there
exist constants $\mu>1$ and $C>0$ such that
$(f^n)'(x) > C\mu^n$ for all $x$ and $n$.

Let $d$ be the degree of $f$. If $I$ is a sufficiently small open interval in $S^1$, 
then for each $n$, $f^{-n}(I)$ is a union of $d$ disjoint intervals.  
Furthermore, each of these intervals has diameter at most $C^{-1}\mu^{-n}$
times the diameter of $I$.  It is now possible to justify 
the assertion in Fundamental Principle $\#$3
in this context.  

\begin{lemma}\label{l.distortion} There exists a constant $K\geq 1$ such that, for all $n\in\NN$, and for
all $x,y\in f^{-n}(I)$, one has:
$$K^{-1}\leq \frac{(f^n)'(x)}{(f^n)'(y)} \leq K.$$
\end{lemma}

\begin{proof} Since $f$ is $C^2$ and $f'$ is bounded away from $0$, 
the function $\alpha(x) = \log(f'(x))$ is $C^1$. In particular, $\alpha$ is
Lipschitz continuous: there exists a constant
$L>0$ such that, for all $x,y\in S^1$,
$|\alpha(x) - \alpha(y)| < L d(x,y)$.
For $n\geq 0$, let $\alpha_n(x) = \log((f^n)'(x))$.
The Chain Rule implies that $\alpha_n(x) = \sum_{i=0}^{n-1} \alpha(f^i(x))$.

The expanding hypothesis on $f$ implies that for all $x,y\in f^{n}(I)$ and for $i=0,\ldots, n$, one
has $d(f^i(x), f^i(y)) \leq C^{-1}\mu^{n-i} d(f^{n}(x),f^{n}(y)) \leq C^{-1}\mu^{n-i}$.
Hence
\begin{eqnarray*}
|\alpha_n(x) - \alpha_n(y)|&\leq & \sum_{i=0}^{n-1} |\alpha(f^i(x)) - \alpha(f^i(y))|\\
&\leq & L \sum_{i=0}^{n-1} d(f^i(x), f^i(y)) \\
&\leq & L \sum_{i=0}^{n-1} C^{-1}\mu^{n-i}\\
& < & LC^{-1}\mu^{-1}(1-\mu^{-1})^{-1}.  
\end{eqnarray*}
Setting $K=\exp(LC^{-1}\mu^{-1}(1-\mu^{-1})^{-1})$, one now sees
that $(f^n)'(x)/(f^n)'(y)$ lies in the interval
$[K^{-1}. K]$, proving the claim.
\end{proof}

In this distortion estimate, the function $\alpha\colon M\to \RR$ is called a {\em cocycle}.  The same
argument applies to any Lipschitz continuous (or even H\"older continuous) cocycle.

\bigskip

\noindent{\bf Ergodicity of expanding maps}

\medskip

The ergodic properties of $C^2$ expanding maps are completely understood.  In particular,
every conservative expanding map is ergodic, and every expanding map is conservative.
The proofs of these facts use Fundamental Principles $\#$1 and 3 in a fairly direct way.

Every $C^2$ conservative expanding map is ergodic with respect to volume.  The proof
is a straightforward adaptation of the proof of Proposition~\ref{p.T_2ergodic} (see, e.g.
\cite{mane}).  Here is a description of the proof for $M=S^1$. 
As  remarked earlier, the proof of Proposition~\ref{p.T_2ergodic} adapts easily
to a general expanding map $f\colon S^1\to S^1$ once one shows 
that for every $f$-invariant set $A$,
and every connected component $J$ of $f^{-n}(S^1\setminus\{p\})$, the quantity 
$$\frac{\lambda(f^{n}(A\cap J)\colon  f^{n}(J))}{\lambda(A\colon J)}$$
is bounded independently of $n$.  This is a fairly
direct consequence of the distortion estimate in Lemma~\ref{l.distortion} and is left as an exercise.

The same distortion estimates show that every $C^2$ expanding map is conservative, preserving a probability measure $\nu$ in the measure class of volume.  
Here is a sketch of the proof for the case $M=S^1$. 
To prove this, consider the push-forward
$\lambda_n = f^n_\ast \lambda$.  Then  $\lambda_n$ is 
equivalent to Lebesgue, and its Radon-Nikodym
derivative $\frac{d\lambda_n}{d\lambda}$ is  the density function
$$\rho_n (x) = \sum_{y\in f^{-n}(x)} \frac{1}{\hbox{jac}_y f^n}.$$ 
Since $f^n_\ast\lambda$ is a probability measure, it follows that $\int_{S^1}\rho_n \,d\lambda=1$.
A simple argument using the distortion estimate above (and summing up over
all $d^n$ branches of $f^{-n}$ at $x$) shows  that 
there exists a constant $c\geq 1$ such that for all $x,y\in S^1$, 
$$c^{-1}\leq \frac{\rho_n(x)}{\rho_n(y)} \leq c.$$

Since the integral of $\rho_n$ is $1$, the functions $\rho_n$ 
are uniformly bounded away from $0$ and $\infty$.  It is easy to see 
that the measure $\nu_n = \frac{1}{n} \sum_{i=1}^n f^i_\ast \lambda$ has density $\frac{1}{n} \sum_{i=1}^n \rho_i$.
Let $\nu$ be any subsequential weak* limit of  $\nu_n$; then $\nu$ is absolutely continuous,
with density $\rho$ bounded away from $0$ and $\infty$. With a little more care, one can show
that $\rho$ is actually Lipschitz continuous.

As a passing comment, the ergodicity of $\nu$ and positivity of $\rho$ imply that $\nu$ is the unique  $f$-invariant measure absolutely continuous with respect to $\lambda$.  With more work,
one can show that $\nu$ is exact.  See \cite{mane} for details.

\bigskip

\noindent{\bf Ergodicity of conservative Anosov diffeomorphisms}

\medskip

Like conservative $C^2$ expanding maps, conservative $C^2$ Anosov diffeomorphisms are ergodic.
This subsection outlines a proof of this fact.
Unlike expanding maps, however, Anosov
diffeomorphisms need not be conservative. The subsection following this one
describe a type of invariant measure that is ``natural'' with respect to volume, called a Sinai-Ruelle-Bowen
(or SRB) measure.  The central result for hyperbolic systems states that every hyperbolic
attractor carries an SRB measure.

\bigskip

\noindent{\em The Hopf Argument}

\medskip

In the 1930's Hopf \cite{hopf} proved that the geodesic flow
for a compact, negatively-curved surface is ergodic.  His method
was to study the Birkhoff averages of continuous functions along leaves
of the stable and unstable foliations of the flow.  This type
of argument has been used since then in increasingly general contexts,
and has come to be known as the Hopf Argument. 

The core of the Hopf Argument is very simple.  To any $f\colon M\to M$ one
can associate the {\em stable equivalence relation $\sim_s$}, where $x\sim_s y$ iff
$\lim_{n\to \infty} d(f^n(x), f^n(y)) = 0$. Denote by $W^s(x)$ the
stable equivalence class containing $x$.  When $f$ is invertible, one defines
the {\em unstable equivalence relation} to be the stable
equivalence relation for $f^{-1}$, and one denotes by $W^u(x)$  the
unstable equivalence class containing $x$. 

The first step in the Hopf Argument is to show that Birkhoff 
averages for continuous functions are
constant along stable and unstable equivalence classes.  Let $\phi\colon M\to \RR$
be an integrable function, and let
\begin{eqnarray}\label{e.birkhoff}
\overline\phi = \limsup_{n\to \infty} \frac{1}{n}\sum_{i=1}^n \phi\circ f^i.
\end{eqnarray}
Observe that if $\phi$ is continuous, then for every $x\in M$ and $x'\in W^s(x)$,  
$\lim_{n\to\infty}|\phi(f^i(x)) -\phi(f^i(x'))| = 0$.  It follows immediately
that $\overline\phi_f(x)=\overline\phi_f(x')$.  In particular, if the limit in
(\ref{e.birkhoff}) exists at $x$, then it exists and is constant on $W^s(x)$.

\bigskip

\noindent{\em  {\bf \em Fundamental Principle $\#$4:} Birkhoff averages of continuous
functions are constant along stable equivalence classes.}

\medskip

The next step of Hopf's argument confines itself to the situation where
$f$ is conservative and Anosov.  In this case, $f$ is invertible, the stable equivalence classes
are precisely the leaves of the stable foliation $\cW^s$, and the unstable
equivalence classes are the leaves of the unstable foliation $\cW^u$.  Since
$f$ is conservative, the Ergodic Theorem
implies that for every $L^2$ function $\phi$, the function $\overline\phi_f$
is equal (mod 0) to the projection of $\phi$ onto the $f$-invariant functions in $L^2$.
Since this projection is continuous, and the continuous functions are dense in $L^2$,
to prove that $f$ is ergodic, it suffices to show that the projection of 
any continuous function is trivial.  That is, it suffices to show that for 
every continuous $\phi$, the function $\overline\phi_f$ is constant (a.e.).  

To this end, let $\phi\colon M\to \RR$ be continuous.  Since the $f$-invariant functions coincide
with the $f^{-1}$-invariant functions, one obtains that $\overline\phi_f = \overline\phi_{f^{-1}}$ a.e.
The previous argument shows $\overline\phi_f$ is constant along $\cW^s$-leaves
and $\overline\phi_{f^{-1}}$ is constant along $\cW^u$-leaves.  The desired 
conclusion is that $\overline\phi_f$ is a.e. constant.  It suffices to 
show this in a local chart, since
the manifold $M$ is connected.  In a local chart, after a smooth change of coordinates, one obtains a pair of transverse foliations $\cF_1$, $\cF_2$ of the cube $[-1,1]^n$ by disks,
and a measurable function $\psi\colon [-1,1]^n \to \RR$ that is constant along the leaves of $\cF_1$
and constant along the leaves of $\cF_2$.

When the foliations $\cF_1$ and $\cF_2$ are smooth (at least $C^1$), one can perform a further
smooth change of coordinates so that $\cF_1$ and $\cF_2$ are transverse coordinate subspace foliations.
In this case, Fubini's theorem implies that any measurable function that is constant along two 
transverse coordinate foliations is a.e. constant.  This completes the proof in the case
that the foliations $\cW^s$ and $\cW^u$ are smooth.  In Hopf's original argument, the stable and
unstable foliations were assumed to be $C^1$ foliations 
(a hypotheses satisfied in the examples he considered, due to low-dimensionality.  See also 
\cite{sinaigeodesic}, where a pinching condition on the curvature, rather than low dimensionality,
implies this $C^1$ condition on the foliations.)

\bigskip

\noindent{\em Absolute continuity}

\medskip

For a general Anosov diffeomorphism or flow, the stable and unstable foliations are not $C^1$, and
so the final step in Hopf's orginal argument does not apply.  The fundamental advance of
Anosov and Anosov-Sinai was to prove that the stable and unstable foliations of an Anosov
diffeomorphism (conservative or not) satisfy a weaker condition than smoothness, called {\em absolute continuity}.
For conservative systems, absolute continuity is enough to finish Hopf's argument, proving that every $C^2$ conservative Anosov diffeomorphism is ergodic \cite{anosov, anosovsinai}.

For a definition and careful discussion of absolute continuity of a foliation $\cF$, see \cite{brinstuck}. 
Two consequences of the absolute continuity of $\cF$ are:
\begin{enumerate}
\item[(AC1)] If $A\subset M$ is any measurable set, then
$$\lambda(A) = 0 \quad\iff\quad \lambda_{\cF(x)}(A) = 0,\quad\hbox{for}\quad \lambda-\hbox{a.e. }\, x\in M,$$
where $\lambda_{\cF(x)}$ denotes the induced Riemannian volume on the leaf of $\cF$ through $x$.
\item[(AC2)] If $\tau$ is any small, smooth disk transverse to a local leaf of $\cF$, and $T\subset \tau$ is
a $0$-set in $\tau$ (with respect to the induced Riemannian volume on $\tau$), then the union of the $\cF$ leaves through points in $T$ has Lebesgue measure $0$ in $M$.
\end{enumerate}

The proof that $\cW^s$ and $\cW^u$ are absolutely continuous has a similar flavor to the proof that an expanding map
has a unique absolutely continuous invariant measure (although the cocycles involved are
H\"older continuous, rather than Lipschitz), and the facts are intimately related.

With absolute continuity of the stable and unstable foliations in hand, one 
can now prove:

\begin{theorem}(Anosov) Let $f$ be a $C^2$, conservative Anosov diffeomorphism.  Then $f$ is ergodic.
\end{theorem}

\begin{proof}[Sketch of proof.]  By the Hopf Argument, it suffices to show that if $\psi^s$
and $\psi^u$ are $L^2$ functions with the following properties:
\begin{enumerate}
\item $\psi^s$ is constant along leaves of $\cW^s$, 
\item $\psi^u$ is constant along leaves of $\cW^u$, and
\item $\psi^s=\psi^u$ a.e.,
\end{enumerate}
then $\psi^s$ (and so $\psi^u$ as well) is constant a.e.

This is proved using the absolute continuity of $\cW^u$ and $\cW^s$. 
Since $M$ is connected, one may argue this locally.  
Let $G$ be the full measure set of $p\in M$ such that 
$\psi^s=\psi^u$. Absolute continuity of $\cW^s$ (more precisely, consequence (AC1) of absolute continuity described above)  implies that for almost every $p\in M$, $G$ has full measure in $\cW^s(p)$.  Pick such a $p$.  Then for almost every $q\in \cW^s(p)$, $\psi^s(q)=\psi^u(p)$;
defining $G'$ to be the union over all $q\in \cW^s(p)\cap G$ of $\cW^u(q)$, 
one obtains that $\psi^s$ is constant on $G\cap G'$.
But now, since $ \cW^s(p)\cap G$ has full measure in $\cW^s(p)$, the 
absolute continuity of $\cW^u$ (consequence (AC2) above) implies that 
$G'$ has full measure in a neighborhood of $p$.  Hence $\psi^s$ 
is a.e. constant in a neighborhood of $p$, completing the proof.
\end{proof}

\bigskip

\noindent{\bf SRB measures}

\medskip

In the absence of a smooth invariant measure, it is still possible for a map to have an
invariant measure that behaves naturally with respect to volume. In computer simulations
one observes such measures when one picks a point $x$ at random and plots many iterates
of $x$; in many systems, the resulting picture is surprisingly insensitive to the initial
choice of $x$.  What appears to be happening in these systems is that the trajectory
of almost every $x$ in an open set $U$ is converging to the support of a singular invariant probability measure
$\mu$.  Furthermore, for any open set $V$, the proportion of forward iterates of $x$ spent in  $V$ appears to
converge to $\mu(V)$ as the number of iterates tends to $\infty$.  

In the 1960's and 70's, Sinai, Ruelle and Bowen rigorously established
the existence of these physically observable measures
for hyperbolic attractors \cite{sinai, ruelle, bowen}. Such measures are now known
as Sinai-Ruelle-Bowen (SRB) measures, and have been shown to exist for non-hyperbolic
maps with some hyperbolic features.  Yhis subsection describes the construction
of SRB measures for hyperbolic attractors.

An $f$-invariant probability measure $\mu$ is called an {\em SRB (or physical) measure} if there
exists an open set $U\subset M$ containing the support of $\mu$ such that, 
for every continuous function $\phi\colon M\to \RR$ and $\lambda$-a.e. 
$x\in U$,
$$\lim_{n\to \infty} \frac{1}{n}\sum_{i=1}^n \phi(f^i(x)) = \int_M f\,d\mu.$$
The maximal open set $U$  with this property is called the {\em basin} of $f$.
To exclude the possibility that the SRB measure is supported on a periodic sink, one often adds the condition that at least one of the Lyapunov 
exponents of $f$ with respect $\mu$ is
positive. Other definitions of SRB measure have been proposed 
(see \cite{youngsrb}).
Note that every ergodic absolutely continuous invariant measure with positive density
in an open set is an SRB measure.  Note also
that an SRB measure for $f$ is not in general an SRB measure for $f^{-1}$, unless
$f$ preserves an ergodic absolutely continuous invariant measure.

Every transitive Anosov diffeomorphism carries a unique SRB measure.  To prove this, one defines a sequence
of probability measures $\nu_n$ on $M$ as follows.  Fix a point 
$p\in M$, and define $\nu_0$ to be the normalized
restriction of Riemannian volume to a ball $B^u$ in $\cW^u(p)$.  Set 
$\nu_n =\frac1n \sum_{i=1}^n f_\ast^i\nu_0$.
Distortion estimates show that the density of $\nu_n$ on its support 
inside $\cW^u$ is bounded, above and below, independently of $n$.
Passing to a subsequential weak* limit, one obtains a probability 
measure $\nu$ on $M$ with bounded densities on $\cW^u$-leaves. 

To show that $\nu$ is an SRB measure, choose a point $q\in M$ in 
the support of $\nu$.  
Since $\nu$ has positive density on unstable manifolds, almost every point in a neighborhood of $q$
in $\cW^u(q)$ is a regular point for $f$ (that is, a point where the forward Birkhoff averages of every continuous
function exist). A variation on the Hopf Argument, using the absolute continuity of $\cW^s$,
shows that $\nu$ is an ergodic SRB measure.

A similar argument shows that every transitive hyperbolic attractor admits an ergodic
SRB measure.  In fact this SRB measure has much stronger mixing properties, namely, it
is Bernoulli. To prove this, one first constructs a Markov partition \cite{bowenmarkov}
conjugating the action of $f$ to a Bernoulli shift.  This map sends the SRB measure to
a Gibbs state for a mixing Markov shift (see the entry on Equilibrium States in Ergodic Theory in this
volume).  A result that subsumes all of the results in this section is:

\begin{theorem}(Sinai, Ruelle, Bowen) 
Let $\Lambda\subset M$ be a transitive hyperbolic attractor for a $C^2$ map $f\colon M\to M$.  Then
$f$ has an ergodic SRB measure $\mu$ supported on $\Lambda$.  Moreover: the disintegration
of $\mu$ along unstable manifolds of $\Lambda$ is equivalent to the induced Riemannian volume,
the Lyapunov exponents of $\mu$ are all positive, and $\mu$ is Bernoulli.
\end{theorem}

\section{Beyond uniform hyperbolicity}

The methods developed in the smooth ergodic theory of hyperbolic maps 
have been extended beyond the hyperbolic context.
Two natural generalizations of hyperbolicity are:
\begin{itemize}
 \item partial hyperbolicity, which requires uniform expansion of $E^u$ and uniform contraction of $E^s$, but allows central directions at each point, in which the expansion and contraction is dominated by the behavior in the hyperbolic directions; and
\item nonuniform hyperbolicity, which requires hyperbolicity along almost every orbit, but allows the expansion of $E^u$ and the contraction of $E^s$ to
weaken near the exceptional set where there is no hyperbolicity.
\end{itemize}
This section discusses both generalizations.

\bigskip

\noindent{\bf Partial hyperbolicity}

\medskip

Brin and Pesin \cite{brinpesin} and independently Pugh and Shub \cite{PSanosovactions}
first examined the ergodic properties of partially hyperbolic systems soon
after the work of Anosov and Sinai on hyperbolic systems. 
One says that a diffeomorphism 
$f\colon M\to M$ of a compact manifold $M$ is {\em partially hyperbolic}
if there is a nontrivial, continuous splitting of the tangent bundle,
$TM=E^s\oplus E^c\oplus E^u$, invariant under $Df$, such that
$E^s$ is uniformly contracted, $E^u$ is uniformly expanded, and $E^c$
is dominated, meaning that for some $n\geq 1$ and for all $x\in M$:
$$\|D_xf^n\vert_{E^s}\| < m(D_xf^n\vert_{E^c})\leq \|D_xf^n\vert_{E^c}\| < m(D_xf^n\vert_{E^u}).
$$

Partial hyperbolicity is a $C^1$-open condition: any diffeomorphism sufficiently $C^1$-close to a partially hyperbolic diffeomorphism is itself partially hyperbolic.  For an extensive discussion of examples of partially hyperbolic dynamical systems, see the survey article \cite{bpsw} and the book \cite{pesinlectures}. Among these examples are: the time-$1$ map of an Anosov flow, the frame flow for a compact manifold of negative sectional curvature, and many affine transformations of compact homogeneous spaces. All of these examples preserve the volume induced by a Riemannian  metric on $M$.  

As in the Anosov case, the stable and unstable bundles $E^s$ and $E^u$ of
a partially hyperbolic diffeomorphism are tangent to foliations, 
again denoted by $\cW^s$ and $\cW^u$ respectively \cite{brinpesin}. 
Brin-Pesin and Pugh-Shub proved that these foliations 
are absolutely continuous.

A partially hyperbolic diffeomorphism $f\colon M\to M$ is {\em accessible} if
any point in $M$ can be reached from any other along an {\em $su$-path}, which
is a concatenation of finitely many subpaths, 
each of which lies entirely in a single leaf of $\cW^s$ or a single leaf of $\cW^u$.
Accessibility is a global, topological property of the foliations
$\cW^u$ and $\cW^s$ that is the analogue of transversality of $\cW^u$ and $\cW^s$
for Anosov diffeomorphisms.  In fact, the transversality of these foliations
in the Anosov case immediately implies that every Anosov diffeomorphism is accessible.
Fundamental Principle $\#$4 suggests that accessibility might be related to
ergodicity for conservative systems.  

\bigskip

\noindent{\em Conservative partially hyperbolic diffeomorphisms}

\medskip

Motivated by a breakthrough result with Grayson \cite{gps},
Pugh and Shub conjectured that accessibility implies ergodicity,
for a $C^2$, partially hyperbolic conservative diffeomorphism \cite{psmontevideo}. 
This conjecture has been proved under the hypothesis of center bunching \cite{bw}, which
is a mild spectral condition on the restriction of $Df$ to the
center bundle $E^c$.  Center bunching is satisfied by most examples of interest, 
including all partially hyperbolic diffeomorphisms with $\hbox{dim}(E^c)=1$. 
The proof in \cite{bw} is a modification of the Hopf Argument using Lebesgue density points
and a delicate analysis of the geometric and measure-theoretic properties of
the stable and unstable foliations.

In the same article, Pugh and Shub also conjectured that 
accessibility is a widespread phenomenon, holding for an open and dense
set (in the $C^r$ topology) of partially hyperbolic diffeomorphisms.  This conjecture has
been proved completely for $r=1$ \cite{dw},  and for all $r$, with the additional assumption
that the central bundle $E^c$ is one dimensional \cite{hhu}.

Together, these two conjectures imply the third, central conjecture: 
in \cite{psmontevideo}:
\begin{conjecture}[Pugh--Shub] For any $r\geq 2$, the $C^r$, conservative partially hyperbolic
diffeomorphisms contain a $C^r$ open and dense set of ergodic diffeomorphisms.
\end{conjecture}
The validity of this conjecture in the absence of center bunching is currently an open question.

\bigskip

\noindent{\em  Dissipative partially hyperbolic diffeomorphisms}

\medskip

There has been some progress in constructing SRB-type measures for dissipative partially
hyperbolic diffeomorphisms, but the theory is less developed than in the
conservative case.  Using the same construction as  for Anosov diffeomorphisms,
one can construct invariant probability measures that are smooth along
the $\cW^u$ foliation \cite{pesinsinai}.  Such measures are referred
to as $u$-Gibbs measures.  Since the stable bundle $E^s$ is not transverse
to the unstable bundle $E^u$, 
the Anosov argument cannot be carried through to show
that $u$-Gibbs measures are SRB measures.  

Nonetheless, there are conditions that imply
that a $u$-Gibbs measure is an SRB measure: for example, a $u$-Gibbs measure
is SRB if it is the unique $u$-Gibbs measure \cite{dimasrb}, if the bundle
$E^s\oplus E^c$ is nonuniformly contracted \cite{bv}, or if the bundle $E^u\oplus E^c$
is nonuniformly expanded \cite{abv}. The proofs of the latter two results use Pesin Theory,
which is explained in the next subsection.  

SRB measures have also been constructed in
systems where $E^c$ is nonuniformly hyperbolic \cite{burnsdolgpesin}, and
in (noninvertible) partially hyperbolic covering maps where $E^c$ is $1$-dimensional \cite{tsujii}.
It is not known whether accessibility plays a role in the existence of SRB measures
for dissipative, non-Anosov partially hyperbolic diffeomorphisms.

\bigskip

\noindent{\bf Nonuniform hyperbolicity}

\medskip

The concept of Lyapunov exponents gives a natural way to extend the
notion of hyperbolicity to systems that
behave hyperbolically, but in a nonuniform manner.  The fundamental principles 
described above, suitably modified, apply to these nonuniformly
hyperbolic systems and allow for the development of a smooth ergodic
theory for these systems.  This program was initially proposed and 
carried out by Yakov Pesin in the 1970's \cite{pesinnuh2}
and has come to be known as {\em Pesin theory}.

Oseledec's Theorem implies that if a smooth map $f$ satisfying the condition
$m(D_xf)>0$ preserves a probability measure $\nu$, then for $\nu$-a.e. $x\in M$
and every nonzero vector $v\in T_xM$, the limit
$$\lambda(x,v) = \lim_{n\to\infty}\frac{1}{n}\sum_{i=1}^n \log\|D_xf^i(v)\|$$
exists.  The number $\lambda(x,v)$ is called the {\em Lyapunov exponent at $x$ in the direction of $v$}.
For each such $x$, there are finitely many possible values for the exponent
$\lambda(x,v)$, and the function $x\mapsto \lambda(x,\cdot)$ is  measurable. 
See the discussion of Oseledec's Theorem in the entry on Ergodic Theorems in this volume.

Let $f$ be a smooth map.  An $f$-invariant probability measure $\mu$ is 
{\em hyperbolic} if the Lyapunov exponents of 
$\mu$-a.e. point are all nonzero. Observe that any invariant measure of a 
hyperbolic map is a hyperbolic measure.

A conservative diffeomorphism $f\colon M\to M$ is {\em nonuniformly hyperbolic} if
the invariant measure equivalent to volume is hyperbolic.
The term ``nonuniform'' is a bit misleading, as uniformly hyperbolic conservative systems
are also nonuniformly hyperbolic. Unlike uniform hyperbolicity, however,
nonuniform hyperbolicity allows for the {\em possibility} of different 
strengths of hyperbolicity along different orbits.

Nonuniformly hyperbolic diffeomorphisms exist on all manifolds \cite{katokbernoulli, dolgpesin}, and
there are $C^1$- open sets of nonuniformly hyperbolic diffeomorphisms that are
not Anosov diffeomorphisms \cite{sw}.  In general, it is a very difficult problem
to establish whether a given map carries a hyperbolic measure  that is nonsingular with
respect to volume.

\bigskip

\noindent{\em Hyperbolic blocks}

\medskip

As mentioned above, the derivative of $f$ along almost every orbit of a nonuniformly hyperbolic system 
looks like the derivative down the orbit of a uniformly hyperbolic system; the
nonuniformity can be detected only by examining a positive measure set of orbits. 
Recall that Lusin's Theorem in measure theory states that every Borel measurable function
is continuous on the complement of an arbitrarily small measure set.
A sort of analogue of Lusin's theorem holds for nonuniformly hyperbolic
maps: every $C^2$, nonuniformly hyperbolic diffeomorphism is uniformly 
hyperbolic on a (noninvariant) compact set whose complement has arbitrarily small
measure.  The precise formulation of this statement is omitted, 
but here are some of its salient features.

If $\mu$ is a hyperbolic measure for a $C^2$ diffeomorphism, then 
attached to $\mu$-a.e. point $x\in M$ are transverse, smooth stable and 
unstable manifolds for $f$. The collection of all stable manifolds is called
the {\em stable lamination} for $f$, and the collection of all unstable manifolds is called
the {\em unstable lamination} for $f$. 
The stable lamination is invariant under $f$, meaning that $f$ sends the stable manifold
at $x$ into the stable manifold for $f(x)$.  The stable manifold through $x$ is contracted
uniformly by all positive iterates of $f$ in a neighborhood of $x$.  Analogous statements
hold for the unstable manifold of $x$, with $f$ replaced by $f^{-1}$.

The following quantites vary measurably in $x\in M$:
\begin{itemize}
\item the (inner) radii of the stable and unstable manifolds through $x$,
\item the angle between stable and unstable manifolds at $x$, and
\item the rates of contraction in these manifolds.
\end{itemize}

The stable and unstable laminations of a nonuniformly hyperbolic system
are absolutely continuous. The precise definition of absolute continuity here is slightly
different than in the uniformly and partially hyperbolic setting, but the consequences
(AC1) and (AC2) of absolute continuity continue to hold. 

\bigskip

\noindent{\em Ergodic properties of nonuniformly hyperbolic diffeomorphisms}

\medskip

Since the stable and unstable laminations are absolutely continuous, the Hopf Argument can
be applied in this setting to show:

\begin{theorem}[Pesin] Let $f$ be $C^2$, conservative and nonuniformly hyperbolic.  
Then there exists a (mod 0) partition $\cP$
of $M$ into countably many $f$-invariant sets of positive volume 
such that the restriction of $f$ to each $P\in\cP$ is ergodic.
\end{theorem} 
The proof of this theorem is also exposited in \cite{psergatt}.
The countable partition can in examples be countably infinite;
nonuniform hyperbolicity alone does not imply ergodicity. 

\bigskip

\noindent{\em The dissipative case}

\medskip

As mentioned above, establishing the existence of a nonsingular hyperbolic
measure is a difficult problem in general.  In systems with some
global form of hyperbolicity, such as partial hyperbolicity,
it is sometimes possible to ``borrow'' the expansion from the unstable
direction and lend it to the central direction, via a small perturbation.
Nonuniformly hyperbolic attractors have been constructed in this way \cite{vianaihes}.  
This method is also behind the construction of a $C^1$ open set of nonuniformly
hyperbolic diffeomorphisms in \cite{sw}.

For a given system of interest, it is sometimes possible to prove that 
a given invariant measure is hyperbolic by establishing an {\em approximate}
form of hyperbolicity.  The idea, due to Wojtkowski and called the {\em cone method},
is to isolate a measurable bundle of cones in $TM$ defined over
the support of the measure, such that the cone at a point $x$ is mapped
by $D_xf$ into the cone at $f(x)$.  Intersecting the images of these cones under
all iterates of $Df$, one obtains
an invariant subbundle of $TM$ over the support of $f$ that is nonuniformly expanded.

Lai-Sang Young has developed a very general method \cite{youngtowers} for proving the
existence of SRB measures with strong mixing properties in systems that display ``some hyperbolicity.''
The idea is to isolate a region $X$ in the manifold where the first return map is
hyperbolic and distortion estimates hold.  If this can be done, then the map
carries a mixing, hyperbolic SRB measure.  The precise rate of mixing is determined by
the properties of the return-time function to $X$; the longer the return times, the slower
the rate of mixing.

More results on the existence of hyperbolic measures  are
discussed in the next section.

An important subject in smooth ergodic theory is the relationship between entropy, Lyapunov
exponents, and dimension of invariant measures of a smooth map.  Significant results in this area
include the Pesin entropy formula \cite{pesinnuh1entropy}, the Ruelle entropy inequality, \cite{ruelleentropy},
the entropy-exponents-dimension formula of Ledrappier-Young \cite{ledyoung1, ledyoung2},
and the proof by Barreira-Pesin-Schmeling that hyperbolic measures have a well-defined dimension \cite{barpessch}. 
The entry on Hyperbolic Dynamical Systems in this volume contains a discussion of these 
results; see this entry there for further information.


\section{The presence of critical points and other singularities}\label{s.critical}

Now for a discussion of the aforementioned technical difficulties that arise
in the presence of singularities and critical points for the derivative.
 
{\em Singularities}, that is, points where $Df$ (or even $f$) fails to be defined, 
arise naturally in the study of billiards and hard sphere gases.  The first
subsection discusses some progress made in smooth ergodic theory in the presence of singularities.

{\em Critical points}, that is, points where $Df$ fails to be invertible, appear 
inescapably in the study of noninvertible maps. This type of complication already shows 
up for noninvertible maps in dimension $1$, in the study of unimodal
maps of the interval. The second subsection discusses 
the technique of parameter exclusion, developed by
Jakobson, which allows for an ergodic analysis of a parametrized family of maps with
criticalities.

The technical advances used to overcome these issues in the interval
have turned out to have applications to dissipative, nonhyperbolic, 
diffeomorphisms in higher dimension, where the derivative is ``nearly critical''
in places.  The last subsection describes extensions of the parameter exclusion
technique to these near-critical maps.

\bigskip

\noindent{\bf Hyperbolic billiards and hard sphere gases}

\medskip

In the 1870's the physicist Ludwig Boltzmann hypothesized that in a mechanical system with many interacting particles, physical measurements (observables), averaged over time, will converge to their expected value as time approaches infinity. The underlying dynamical system in this statement is a Hamiltonian system with many
degrees of freedom, and the ``expected value''  is with respect to Liouville measure. 
Loosely phrased in modern terms, Boltzmann's hypothesis states that a generic Hamiltonian system of
this form will be ergodic on constant energy submanifolds.   Reasoning  that the time scales involved in measurement of
an observable in such a system are much larger than the rate of evolution of the system, Boltzmann's
hypothesis allowed him to assume that physical quantities associated to such a system 
behave like constants.

In 1963, Sinai revived and formalized this ergodic hypothesis, stating it in a concrete
formulation known as the Boltzmann-Sinai Ergodic Hypothesis.  In Sinai's formulation, the
particles were replaced by $N$ hard, elastic spheres, and to compactify the problem, he
situated the spheres on a $k$-torus, $k=2,3$.  The  Boltzmann-Sinai Ergodic Hypothesis is the conjecture
that the induced Hamiltonian system on the $2kN$-dimensional configuration space is ergodic
on constant energy manifolds, for any $N\geq 2$.  

Sinai verified this conjecture for $N=2$ by reducing the problem to a billiard map in the plane.  As background for Sinai's result, a brief discussion of
planar billiard maps follows.

Let $D\subset \RR^s$ be a connected region whose boundary $\partial D$ is a collection of closed, piecewise smooth simple curves  the plane. The {\em billiard map} is a map defined (almost everywhere)
on $\partial D \times [-\pi,\pi]$.  To define this map, one identifies each point  
$(x,\theta)\in \partial D \times [-\pi,\pi]$ with an inward-pointing tangent vector at
$x$ in the plane, so that the normal vector to $\partial D$ at $x$ corresponds to
the pair $(x,\pi/2)$.  This can be done in a unique way on the smooth components of $\partial D$.
Then $f(x, \theta)$ is obtained by following the ray originating at $(x,\theta)$ until
it strikes the boundary $\partial D$ for the first time at $(x',\theta')$.  Reflecting
this vector about the normal at $x'$, define $f(x,\theta) = (x',\pi-\theta')$.

It is not hard to see that the billiard map is conservative.  The billiard map is
piecewise smooth, but not in general smooth: the degree of smoothness of $f$ is one
less than the degree of smoothness of $\partial D$.  In addition to singularities
arising from the corners of the table, there are singularities arising in the
{\em second} derivative of $f$ at the tangent vectors to the boundary. 

In the billiards
studied studied by Sinai, the boundary $\partial D$ consists of a union
of concave circular arcs and straight line segments.  Similar billiards,
but with convex circular arcs, were first studied by Bunimovich \cite{bunimovich}.
Sinai and Bunimovich proved that these billiards are ergodic and
nonuniformly hyperbolic.
For the Boltzmann-Sinai problem with $N\geq 3$, the relevant associated
dynamical system is a higher dimensional billiard table in euclidean space,
with circular arcs replaced by cylindrical boundary components.  

In a planar billiard table with circular/flat boundary, the behavior of vectors encountering
a flat segment of boundary is easily understood, as is the behavior of vectors
meeting a circular segment in a neighborhood of the normal vector. 
If the billiard map is ergodic, however, every open
set of vectors will meet the singularities in the table infinitely many times. 
To establish the nonuniform hyperbolicity  of such billiard tables via conefieds, 
it is therefore necessary to understand precisely the fraction of time
orbits spend near these singularities. Furthermore, to use the Hopf argument
to establish ergodicity, one must avoid the singularities in the second derivative,
where distortion estimates break down. The techniques for overcoming
these obstacles involve imposing restrictions on the geometry of the 
table (even more so for higher dimensional tables),
and are well beyond the scope of this paper. 

The study of hyperbolic billiards and hard sphere gases has a long and involved history.
See the articles \cite{szasz} and \cite{chermark} for a survey of some of the results
and techniques in the area.  A discussion of methods in singular smooth ergodic theory,
with particular applications to the Lorentz attractor, can be found in \cite{arapac}.  Another, more classical, 
reference is \cite{kslp}, which contains a formulation
of properties on a critical set, due to Katok-Strelcyn, that are useful
in establishing ergodicity of systems with singularities.

\bigskip

\noindent{\bf Interval maps and parameter exclusion}

\medskip

The logistic family of maps $f_t\colon x\mapsto tx(1-x)$ defined on the interval $[0,1]$ is
very simple to define but exhibits an astonishing variety of dynamical features
as the parameter $t$ varies.  For small positive values of $t$, almost every point in $I$
is attracted under the map $f_t$ to the sink at $-1$.  For values of $t>4$, the map
has a repelling hyperbolic Cantor set. As the value of $t$ increases between $0$ and $4$, the
map $f_t$ undergoes a cascade of period-doubling bifurcations, in which a periodic
sink of period $2^n$ becomes repelling and a new sink of period $2^{n+1}$ is born. At
the accumulation of period doubling at $t\approx 3.57$, a periodic point of period $3$
appears, forcing the existence of periodic points of all periods.  The dynamics of $f_t$ for
$t$ close to $4$ has been the subject of intense inquiry in the last 20 years.

The map $f_t$, for $t$ close to $4$, shares some of the features of the doubling map $T_2$; 
it is $2$-to-$1$, except at the critical point $\frac12$, and it is uniformly expanding
in the complement of a neighborhood of this critical point. Because this neighborhood
of the critical point is not invariant, however, the only invariant sets on which $f_t$
is uniformly hyperbolic have measure zero.  Furthermore, the second derivative
of $f_t$ vanishes at the critical point, making it impossible to control
distortion for orbits that spend too much time near the critical point.

Despite these serious obstacles, Michael Jakobson \cite{jakobson} found 
a method for constructing absolutely continuous invariant measures for maps
in the logistic family.  The method has come to be known as {\em parameter exclusion}
and has seen application far beyond the logistic family.
As with billiards, it is possible to formulate geometric conditions on the map $f_t$
that control both expansion (hyperbolicity) and distortion on a positive measure set.
As these conditions involve understanding infinitely many iterates of $f_t$,
they are impossible to verify for a given parameter value $t$.

Using an inductive formulation of this condition, Jakobson showed that
the set of parameters $t$ near $4$ that fail to satisfy the condition at
iterate $n$ have exponentially small measure (in $n$).  He thereby showed
that for a positive Lebesgue measure set of parameter values $t$, the map $f_t$ has an
absolutely continuous invariant measure \cite{jakobson}.  This measure
is ergodic (mixing) and has a positive Lyapunov exponent.  The delicacy 
of Jakobson's approach is confirmed by the fact that for an open and dense 
set of parameter values, almost every orbit is attracted to
a periodic sink, and so  $f_t$ has no absolutely
continuous invariant measure \cite{swiatek, lyubich}.  Jakobson's method
applies not only to the logistic family but to a very general class
of $C^3$ one-parameter families of maps on the interval.

\bigskip

\noindent{\bf Near-critical diffeomorphisms}

\medskip

Jakobson's method in one dimension proved to extend to certain highly dissipative diffeomorphisms.
The seminal paper in this extension is due to Benedicks and Carleson; the method
has since been extended in a series of papers \cite{moraviana, benyoung, benvi}
and has been formulated in an abstract setting \cite{wangyoung}. 
 
This extension turns out to be highly nontrivial, but it is possible to 
describe informally the similarities
between the logistic family and higher-dimensional ``near critical'' diffeomorphisms. 
The diffeomorphisms to which this method applies are crudely hyperbolic with
a one dimensional unstable direction.  Roughly this means that in some invariant
region of the manifold, the image of a small ball under $f$ will be stretched
significantly in one direction and shrunk in all other directions.  The
directions of stretching and contraction are transverse in a large proportion
of the invariant region, but there are isolated ``near critical'' subregions where
expanding and contracting directions are nearly tangent.  

The dynamics of such a diffeomorphism
are very close to $1$-dimensional if the contraction is strong enough, and
the diffeomorphism resembles an interval map with isolated critcal points,
the critical points corresponding to the critical regions where stable
and unstable directions are tangent.

An illustration of this type of dynamics is the H\'enon family of 
maps $f_{a,b}\colon (x,y)\mapsto (1-ax^2+by, x)$,  the original
object of study in Benedicks-Carlesson's work. When the parameter
$b$ is set to $0$, the map $f_{a,b}$ is no longer a diffeomorphism, and
indeed is precisely a projection composed with the logisitic map.
For small values of $b$ and appropriate values of $a$, the H\'enon
map  is strongly dissipative and displays the near critical behavior
described in the previous paragraph.  In analogy to Jakobson's result,
there is a positive measure set of parameters near $b=0$ where
$f_{a,b}$ has a mixing, hyperbolic SRB measure.

See \cite{luzzatoviana} for a detailed exposition of the parameter exclusion method for H\'enon-like maps.

\section{Future directions}

In addition to the open problems discussed in the previous sections, there
are several general questions and problems worth mentioning: 
\begin{itemize}
\item What can be said about systems with everywhere
vanishing Lyapunov exponents?  Open sets of such systems
exist in arbitrary dimension.  Pesin theory carries into 
the nonuniformly hyperbolic setting the basic principles 
from uniformly hyperbolic theory (in particular, 
Fundamental Principles $\#3$ and $4$ above).  To what 
extent do properties of isometric and unipotent systems (for example,
Fundamental Principle $\#2$) extend to conservative 
systems all of whose Lyaponov exponents vanish?  
\item  Can one establish the existence of and analyze in general
the conservative systems on surfaces that have two positive measure
regimes: one where Lyapunov exponents vanish, and the other where they are 
nonzero?  Such systems are conjectured exist in the presence of KAM 
phenomena surrounding elliptic periodic points.
\item  On a related note, how common are conservative systems whose Lyapunov 
exponents are nonvanishing on a positive measure set?  See \cite{vianaexp}
for a discussion.
\item Find a broad description of those dissipative systems that 
admit finitely (or countably) many physical measures.  Are such systems 
dense among all dissipative systems, or possibly generic among a restricted 
class of systems?  See \cite{sw, palis} for several questions and conjectures
related to this problem.
\item Extend the methods in the study of systems with singularities
to other specific systems of interest, including the infinite dimensional
systems that arise in the study of partial differential equations.  
\item Carry the methods of smooth ergodic theory further into the study
of smooth actions of discrete groups (other than the integers) on manifolds.
When do such actions admit (possibly non-invariant) ``physical'' measures?
\end{itemize}
There are other interesting open areas of future inquiry, but this gives a 
good sample of the range of possibilities.

\end{document}